\documentclass[11pt]{amsart}

\usepackage[english]{babel}
\usepackage[T1]{fontenc}
\usepackage{amsmath,amssymb}
\usepackage[all]{xy}
\usepackage{graphicx}
\usepackage{color}

\newtheorem{definition}{Definition }[section]

\newtheorem{lemma}[definition]
{Lemma}
\newtheorem{theorem}[definition]
{Theorem }
\newtheorem{ex}[definition]
{Example }

\newtheorem{corollary}[definition]
{Corollary}
\newtheorem{prop}[definition]{Proposition}

\newtheorem{rmk}[definition]{Remark}

\newcommand{\lgw}{\longrightarrow}

\newcommand{\s}{\sigma}
\newcommand{\ovl}{\overline}

\renewcommand{\deg}{\text{deg}\,}

\renewcommand{\O}{\mathcal{O}}

\newcommand{\m}{\mathfrak{m}}

\newcommand{\Z}{\mathbb{Z}}

\newcommand{\R}{\mathbb{R}}
\newcommand{\K}{\mathbb{K}}

\newcommand{\N}{\mathbb{N}}
\newcommand{\D}{\Delta}

\newcommand{\C}{\mathbb{C}}
\newcommand{\Q}{\mathbb{Q}}
\newcommand{\Supp}{\text{Supp}}

\renewcommand{\a}{\alpha}
\renewcommand{\b}{\beta}
\newcommand{\g}{\gamma}

\renewcommand{\phi}{\varphi}

\newcommand{\sr}[1]%
{\ifmmode{}^\dagger\else${}^\dagger$\fi\ifvmode
\vbox to 0pt{\vss
 \hbox to 0pt{\hskip\hsize\hskip1em
 \vbox{\hsize3cm\raggedright\pretolerance10000
 \noindent #1\hfill}\hss}\vss}\else
 \vadjust{\vbox to0pt{\vss%
 \hbox to 0pt{\hskip\hsize\hskip1em%
 \vbox{\hsize3cm\raggedright\pretolerance10000%
 \noindent #1\hfill}\hss}\vss}}\fi%
}

\DeclareMathOperator{\NP}{NP}

\begin{document}
\title{The Abhyankar-Jung Theorem}
\author{Adam Parusi\'nski, Guillaume Rond}

\address {}

\email{}

\address {D\'epartement de Math\'ematiques,
Universit\'e de Nice - Sophia Antipolis,
Parc Valrose,
06108 Nice Cedex 02,
France}
\email{adam.parusinski@unice.fr}
\address{Institut de Math\'ematiques de Luminy, 
Universit\'e de la M\'editerran\'ee, Campus de Luminy, Case 907,
13288 Marseille Cedex 9}
\email{guillaume.rond@univmed.fr}

\begin{abstract}
 We show that every quasi-ordinary Weierstrass polynomial 
 $P(Z) = Z^d+a_1 (X) Z^{d-1}+\cdots+a_d(X) \in \K[[X]][Z] $, $X=(X_1, \ldots , X_n)$, over an algebraically 
 closed field of characteristic zero $\K$, such that $a_1=0$, is $\nu$-quasi-ordinary.   That means that 
 if the discriminant    
 $\Delta_P \in \K[[X]]$ is equal to a monomial times a unit  then the ideal $(a_i^{d!/i}(X))_{i=2,\ldots ,d}$ 
 is monomial and generated by one of  $a_i^{d!/i}(X)$.  
 
 We use this result to give a constructive proof of the Abhyankar-Jung Theorem that works for any 
 Henselian local subring of $\K[[X]]$ and  the function germs of quasi-analytic families.  
\end{abstract}
\maketitle

\section{Introduction}

Let $\K$ be an algebraically closed field of characteristic zero and let 
\begin{align}\label{polynomial}
P(Z)=Z^d+a_1(X_1, \ldots , X_n) Z^{d-1}+\cdots+a_d(X_1, \ldots , X_n) \in \K[[X]][Z]
\end{align}
be a unitary polynomial with coefficients formal power series in $X= (X_1, \ldots X_n)$.   Such a polynomial  $P$ is called \emph{quasi-ordinary} if its  discriminant $\Delta_P (X)$ 
equals $X_1^{\alpha_1} \cdots X_n^{\alpha_n} U(X)$, with  $\alpha_i \in \N$ and $U(0)\ne 0$.   We call $P(Z)$ \emph{a Weierstrass polynomial}  if $a_i(0)=0$ for all $i=1, \ldots,d$.   

We show the following result.

\begin{theorem} \label{maintheorem}
Let $\K$ be an algebraically closed field of characteristic zero and let 
$P \in \K[[X]][Z]$ be a quasi-ordinary Weierstrass polynomial such that  $a_1=0$.   
Then  the ideal $(a_i^{d!/i}(X))_{i=2,\ldots ,d}$  is monomial and generated by one of  $a_i^{d!/i}(X)$.  
\end{theorem}

The latter condition is equivalent to  $P$ being $\nu$-quasi-ordinary in the sense of Hironaka, 
\cite{H}, \cite{L}, and satisfying $a_1=0$.  Being $\nu$-quasi-ordinary is a condition on 
the Newton polyhedron of $P$ that we recall in Section \ref{maintheoremproof} below.  
Thus Theorem \ref{maintheorem} can be rephrased as follows. 

\begin{theorem}[\cite{L} Theorem 1] \label{Luengotheorem}
If $P$ is a quasi-ordinary Weierstrass polynomial with $a_1=0$ then $P$ is  $\nu$-quasi-ordinary.
\end{theorem}

As noticed  in \cite{K-V}, Luengo's proof of Theorem \ref{Luengotheorem} is not complete.  We complete the proof of Luengo and thus we complete his proof of the Abhyankar-Jung Theorem.   

\begin{theorem}[Abhyankar-Jung Theorem]\label{AJ}
Let $\K$ be an algebraically closed field of characteristic zero and let 
$P \in \K[[X]][Z]$ be a quasi-ordinary Weierstrass polynomial 
such that the  discriminant of $P$ satisfies 
$\Delta_P (X)= X_1^{\alpha_1} \cdots X_r^{\alpha_r} U(X)$, where 
$U(0)\ne 0$, and $r\le n$.  
Then there is $q\in\N\setminus \{0\}$ such that $P(Z)$ has its roots in 
$\K[[X_1^{\frac{1}{q}},...,X_r^{\frac{1}{q}}, X_{r+1},...,X_n]]$.
\end{theorem} 

Theorem \ref{AJ} has first been proven by Jung in 1908 for $n=2$ and $\K=\C$  in order to give a local uniformisation   of singular complex analytic surfaces \cite{J}.  His method has been then used by Walker \cite{W} and Zariski \cite{Z} to give proofs of resolution of singularities of surfaces, see \cite{PP} for a detailed account of the Jung's method of resolution of singularities of complex surfaces.  The first complete proof of Theorem \ref{AJ} 
appeared in \cite{Ab}.  As shown in \cite{L}, Theorem \ref{AJ} follows fairly easily from his Theorem 1 (our Theorem \ref{Luengotheorem}).  Since then there were other proofs of Theorem \ref{AJ} based on 
Theorem 1 of \cite{L}, see e.g. \cite{Zu}.  

Theorem \ref{maintheorem} is proven in Section \ref{maintheoremproof}.  
In Section \ref{applications} we show how Theorem \ref{maintheorem} gives a procedure to compute the 
 roots of $P$, similar to the Newton algorithm for $n=1$ (as done in \cite{B-M2}), and thus implies the Abhyankar-Jung Theorem. 
 Unlike the one in \cite{L}, our procedure does not use the Weierstrass Preparation Theorem, but only the Implicit Function Theorem.   Thanks to this we are able to  extend the Abhyankar-Jung Theorem to 
 Henselian subrings of $\K[[X]]$, 
 and  quasi-analytic families of function germs 
answering  thus a question posed in \cite{R}.  A similar proof of this latter result was given in \cite{N1} assuming Theorem 1 of \cite{L}. In \cite{N2} is also given a proof of the Abhyankar-Jung Theorem for excellent Henselian subrings of $\K[[X]]$ using model theoretic methods and Artin Approximation.   

It is not difficult to see that Theorem  \ref{maintheorem} and the Abhyankar-Jung Theorem are equivalent, one implies easily the other.  As we mentioned above Theorem  \ref{maintheorem} 
gives the Abhyankar-Jung Theorem.  We show in Section \ref{alternativeproof} how Theorem \ref{maintheorem} can be proven using the Abhyankar-Jung Theorem, see also \cite{Z}.  
We also give  in Section \ref{alternativeproof} an alternative proof of Theorem  \ref{maintheorem} that uses
 the complex analytic version of the Abhyankar-Jung Theorem and 
the Artin Approximation Theorem.  

Finally, in Section \ref{toric}, we extend Abhyankar-Jung Theorem to the toric case following our alternative proof of Theorem \ref{alternativeproof} and using a complex analytic version of the Abhyankar-Jung Theorem in the toric case proven by P. Gonz\'alez P\'erez \cite{G-P1}.

\begin{rmk}
Neither in  Theorem  \ref{maintheorem} nor in the Abhyankar-Jung Theorem the assumption that 
$P$ is Weierstrass is necessary.  Moreover, in Theorem  \ref{maintheorem} the assumption 
that $\K$ is algebraically closed is not necessary.  If  $\K$ is not algebraically closed 
then the roots of $P$ may have coefficients in a finite  extension of $\K$, see Proposition \ref{lastprop} below.  
\end{rmk}

\noindent
 \textbf{Notation. } 
The set of natural numbers including zero is denoted by $\N$.  We denote  $\Q_{\ge 0}= \{x\in \Q; x\ge 0\}$ and 
 $\Q_+= \{x\in \Q; x> 0\}$.  Similarly, by 
 $\R_{\ge 0}$ we denote 
the set $\{x\in \R; x\ge 0\}$.  \\

\noindent
\textbf{Acknowledgment. }
The authors would like to thank P. Gonz\'alez P\'erez  for his valuable suggestions, in particular concerning the toric case.


\section{Preliminary results}

The following proposition is well known, see for instance \cite{S} or \cite{N2}.  We present its proof for the reader's convenience.

\begin{prop}\label{complex}
The Abhyankar-Jung Theorem holds
 for quasi-ordinary polynomials with  complex analytic coefficients, $P\in \C\{X\}[Z]$.  
 \end{prop}
 
\begin{proof}
Fix a polydisc 
$U= \prod_{i=1}^n D_\varepsilon= \{X\in \C^n ; |X_i|<\varepsilon, i=1, \ldots , n \}$ 
such that the coefficients 
$a_i(X)$ of $P$ are analytic on a neighbourhood of $\overline U$.   
   By assumption, 
the projection of  $\{(X,Z) \in U \times \C; P(X,Z) = 0\}$ 
  onto $U$ is a finite branched covering.  
Its restriction over  $U^* = \{X \in U ; X_i\ne 0, i=1, \ldots , r \}$ 
is a finite covering of degree $d$.  Thus there is a substitution of powers  
$$ X(Y) = (Y_1^q, \ldots ,Y_r^q, Y_{r+1}, \ldots , Y_n): U_1 \to U,$$
where  $U_1= \prod_{i=1}^r D_{\varepsilon^{1/q}}\times \prod_{i=r+1}^n D_\varepsilon$,  such that the induced covering over $U_1^*= \prod_{i=1}^r D^*_{\varepsilon^{1/q}}\times \prod_{i=r+1}^n D_\varepsilon$ is trivial.  That is to say on $U_1^*$,   $P(X(Y), Z)$ factors  
$$
P(X(Y), Z) =  \prod (Z - f_i(Y)),
$$
with $f_i $ complex analytic and bounded on $U_1^*$.  \footnote{Fix $u_0\in U^*$.  The fundamental group $\pi_1(U^*, u_0) $ is equal to $\Z^r$.  To each connected finite covering $h:\tilde U^* \to U^*$ and each $\tilde u_0 \in h^{-1} (u_0)$ corresponds 
 a subgroup $h_*(\pi_1(\tilde U^*,\tilde u_0)) \subset \pi_1(U^*,u_0)$ of finite index.  If 
$(q\Z)^r \subset h_*(\pi_1(\tilde U^*,\tilde u_0))$, then the  covering corresponding to $(q\Z)^r \subset \Z^r$, that is the substitution of powers $X(Y): U_1\to U$,
 factors through $h$.  That is there exists an analytic map $\tilde U^* \to \{(X,Z) \in U^* \times \C; P(X,Z) = 0\}$, 
 of the form $Y\to (X(Y), Z(Y))$.  This $Z(Y)$ is one of the functions $f_i$.  
 If we apply this argument to each connected component $\tilde U^*$ of 
 $ \{(X,Z) \in U^* \times \C; P(X,Z) = 0\}$ and to each point of the fiber over $u_0$ we obtain $d$ distinct analytic functions $f_i$.}
 
 Hence,  
by Riemann Removable Singularity Theorem, see e.g. \cite{GR} Theorem 3 page 19, each $f_i$ extends  
to an analytic function on $U_1$.  
\end{proof}

\begin{rmk}\label{rmk_global}
If the coefficients of $P(Z)$ are global analytic functions, and $\Delta_P(X) = X^{\alpha} u(X)$ globally, where $u(X)$ is nowhere vanishing in $\C$,  then we can choose $U=\C^n$ in the former proof. Thus, using the notations of this proof, we see that after a substitution of powers $X_i=Y_i^q$, for $1\leq i\leq r$, we may assume that the roots of $P(Z)$ are global analytic.
\end{rmk}

Given a polynomial $P(X_1, \ldots, X_n,Z) \in \K[X_1, \ldots, X_n,Z]$, where $\K$  is a field of characteristic zero, denote by  $\K_1$ the field generated by the  coefficients of $P$.   Since $\K_1$ is finitely generated over $\Q$ there exists a field embedding $\K_1\hookrightarrow \C$.  
This allows us to extend 
some results from complex polynomials to the polynomials over $\K$.  This is a special case of the Lefschetz principle.  We shall need later two such results.

\begin{prop}\label{restriction}
Let  
\begin{align}\label{polynomialpolynomial}
P(Z)=Z^d+a_1(X) Z^{d-1}+\cdots+a_d(X) \in \K[X][Z] 
\end{align}
be quasi-ordinary (as a polynomial with coefficients in $\K[[X]]$).   Then $(P_{|X_n=0})_{red}$ is quasi-ordinary.  Moreover, the discriminant of $(P_{|X_n=0})_{red}$ divides the discriminant of $P$.  
\end{prop}

\begin{proof}
Denote $Q(X', Z) = P(X_1, X_2, \ldots, X_{n-1}, 0,Z)$, where $X'= (X_1, X_2, \ldots, X_{n-1})$.  Let 
$Q= \prod Q_i^{m_i}$ be the factorisation into irreducible factors. 
Then $(P_{|X_n=0})_{red}= \prod Q_i $.  We may assume that  $P$ and 
each of the $Q_i$'s are defined over a subfield of $\C$.   Thus, by Proposition \ref{complex},   
 the roots  of $P$ are complex 
analytic after a substitution of powers, we write them as  $Z_1(X),\ldots, Z_d(X)\in\C \{X_1^{1/q}, \ldots X_n^{1/q}\}$ for some $q\in\N$.  Since $\Delta_P (X) = \prod _{i\ne j} (Z_i (X) -Z_j(X))$   
$$ 
Z_{i,j}(X) = Z_i(X) - Z_j(X) = X^{\beta_{ij}} u_{ij} (X) , 
$$
where $\beta_{ij} \in \Q_{\ge 0}$, $u_{ij} \in\C \{X_1^{1/q}, \ldots, X_n^{1/q}\}$, 
$u_{ij}(0) \ne 0$. 
Taking $X_n=0$ we see that the differences of the roots of  $Q_{red}= \prod Q_i$  are 
the restrictions $Z_{ij}|_{X_n=0}$ and hence their product is a monomial times a unit, that is $Q_{red}$ 
is quasi-ordinary. 
\end{proof}

\begin{prop}[\cite{L}, Proposition 1]  \label{prop2} 
Let $P\in \K[X][Z]$ be a polynomial of the form \eqref{polynomialpolynomial} 
such that $a_1=0$ and 
   the discriminant  $\Delta_P (X) = c_0 X^\alpha$, $c_0\in \K \setminus 0$, 
$\alpha \ne 0$.  Then, for each $i= 2, 3, \ldots , d$, $a_i(X) = c_i X^{ i  \alpha/{d(d-1)} }$, $c_i\in \K$.   
\end{prop}

\begin{proof}
 Consider first the case $n=1$, $\K=\C$.  By Proposition 
\ref{complex} and Remark \ref{rmk_global}, after a substitution of powers $X=Y^q$, there are analytic functions $f_i(Y), i=1, \ldots , d$, such that 
$$P(Y^q,Z)=  \prod _i (Z-f_i(Y)) . 
$$
As a root of a polynomial  each $f_i$ satisfies 
$$
|f_i(Y)|\le C(1 +|Y|^N)
$$
for $C,N \in \R$, see e.g. \cite{BR}, 1.2.1.  Hence, by Liouville's Theorem, cf. \cite{Ti}, Section 2.52, p. 85, $f_i$  is a polynomial.  By assumption,  
$\Delta_P(Y^q) =c_0 Y^{q\alpha }$, and hence each difference $f_i -f_j$ is a monomial.  For $i,j,k$ distinct we have 
$(f_i-f_j) + (f_j-f_k)+(f_k-f_i)=0$, and therefore all these  monomials should have the same exponent 
$(f_i -f_j)(Y) = c_{i,j} Y^\beta$, where $\beta = \frac {q} {d(d-1)}  \alpha$.  Finally, since $a_1=0$,  each $f_i$ is a monomial  :
$$
f_i = \sum_j \frac {f_i -f_j} d . 
$$

In the general case we consider $P\in \K[X_1,\ldots , X_n][Z]$ as a polynomial in $X_n,Z$ with coefficients in 
$\K'=\K(X_1, \ldots, X_{n-1})$ and $\K' \hookrightarrow \C$.   
Therefore, for every $i$, $a_i$ equals  $X_n^{ i  \alpha_n/{d(d-1)} }$ times a constant of the algebraic closure of $\K'$.    Since $a_i$ is a polynomial in $(X' , X_n)$ it must be equal to $X_n^{ i  \alpha_n/{d(d-1)} }$ times a polynomial in $X'$.  Applying this argument to each variable $X_j$, $j=1, \ldots ,
n$, we see that $a_i$ is the product of all $X_j^{ i  \alpha_j/{d(d-1)} }$ and a constant of $\K$.    This ends the proof. 
\end{proof}


\section{1st proof of Theorem \ref{maintheorem}}\label{maintheoremproof}

Given $P\in \K[[X_1, \ldots ,X_n,Z]]$.  Write 
$$
P(X,Z) = \sum_{(i_1, \ldots, i_{n+1})} P_{i_1, \ldots, i_{n+1}} X^{i_1} \cdots X^{i_n} Z^{i_{n+1}}.
$$
Let $H(P) = \{ (i_1, \ldots, i_{n+1}) \in \N^{n+1}; P_{i_1, \ldots, i_{n+1}} \ne 0\}$.  The Newton polyhedron of $P$ 
is the convex hull in $\R^{n+1}$ of $\bigcup_{a\in H(P)} (a+\R_{\ge 0}^{n+1})$, and we will denote it by $\NP (P)$.  

A Weierstrass polynomial \eqref{polynomial} is called  
\emph {$\nu$-quasi-ordinary} if there is a point $R_1$ of the Newton polyhedron $\NP (P)$, 
$R_1\ne R_0=(0,\ldots, 0,d)$, such that if $R_1'$ denotes the projection of $R_1$ onto $\R^n\times 0$ 
from $R_0$, and $S=|R_0,R_1'|$ is the segment joining $R_0$ and $R'_1$, then 
\begin{enumerate}
\item
$\NP (P) \subset |S| = \bigcup_{s\in S} (s+\R_{\ge 0}^{n+1})$
\item
$P_S= \sum_{(i_1, \ldots, i_{n+1})\in S} P_{i_1, \ldots, i_{n+1}}  X_1^{i_1} \cdots X_n^{i_n} Z^{i_{n+1}}$ is not a power of a 
linear form in $Z$.  
\end{enumerate} 
The second condition is satisfied automatically if $a_1=0$.   

\begin{lemma}\label{lemma3.1}
Let $P(Z)\in \K[[X]][Z]$ be a Weierstrass polynomial \eqref{polynomial} such that $a_1=0$.  
The following conditions are equivalent: 
\begin{enumerate}
\item
$P$ is $\nu$-quasi-ordinary.
\item
$\NP(P)$  has only one compact edge containing $R_0$.  
\item
the ideal $(a_i^{d!/i}(X))_{i=2,\ldots ,d}\subset \K[[X]]$ is monomial and generated by one of  $a_i^{d!/i}(X)$.  
\end{enumerate} 
\end{lemma}

\begin{rmk}
The ideal  $(a_i^{d!/i}(X))_{i=2,\ldots ,d}\subset \K[[X]]$ is exactly the idealistic exponent introduced by Hironaka (see \cite{H2}).
\end{rmk}

\begin{proof}[Proof of Lemma \ref{lemma3.1}]
(3) holds if and only if there is $\gamma\in \N^n$ and 
$i_0 \in \{2,\ldots ,d\}$ such 
that $a_{i_0} (X) = X^\gamma U_{i_0}(X)$, $U_{i_0}(0) \ne 0$, and for all $j\in \{2,\ldots ,d\}$, $X^{j\gamma}$ divides 
$a_j^{i_0}$.  Thus we may take $R_1 = (\gamma, d-i_0)$ and conversely by this formula 
$R_1$ defines $i_0$ and $\gamma$.  Thus (1) is equivalent to (3).  

Let us denote by $\pi_0$ the projection from $R_0$ onto $\R^n\times 0$.  Then both (1) and (2) are 
equivalent to $\pi_0(\NP (P))$ being of the form $p+ \R_{\ge 0}^n$ for some $p\in \R^n$.    
\end{proof}

Let $P(Z)\in\K[[X]][Z]$ be a quasi-ordinary polynomial of degree $d$ with $a_1=0$. 
 Let us assume that $P(Z)$ is not $\nu$-quasi-ordinary.  
 Then, as shown in the proof of Theorem 1 of \cite{L}, p. 403, there exists 
$\b=(\b_1,\ldots ,\b_{n+1})\in(\N\setminus \{0\})^{n+1}$ such that 
\begin{enumerate}
\item $L(u):=\b_1u_1+\cdots+\b_{n+1} u_{n+1}-d\b_{n+1}=0$ is the equation of a hyperplane $H$ of $\N^{n+1}$ containing $(0,...,0,d)$, 
\item $H\cap \NP(P)$  is a compact face of $\NP (P)$ of dimension $\ge 2$,
\item $L(\NP(P))\geq 0$.
\end{enumerate}
The existence of such $\b$ can be also shown as follows.  
Each $\b \in \R_{\ge 0}^{n+1}$ defines a face $\Gamma_\b$ of $ \NP(P)$ by
$$
\Gamma_\b = \{ v\in  \NP(P); <\b,v> = \min _{u \in  \NP(P)} <\b,u> \}.
$$
Each face of $ \NP(P)$ can be obtained this way.  Moreover, since the vertices of $ \NP(P)$ have 
integer coefficients, each face can be defined by $\b \in \Q_{\ge 0}^{n+1}$ and even $\b\in\N^{n+1}$ by multiplying it by an integer.  
If one of the coordinates of $\beta $ is zero then $\Gamma_\b$ is not compact.  Thus it suffices to take as 
$\b$ a vector in $(\N\setminus \{0\})^{n+1}$ 
defining a compact  face containing $R_0$ and of dimension $\ge 2$.

Let 
$$P_H(X,Z):=\sum_{i_1,...,i_{n+1}\in H}P_{i_1,...,i_{n+1}}X_1^{i_1}...X_n^{i_n}Z^{i_{n+1}}. $$ 
and define $\tilde P_H$ as $P_H$ reduced.  
 If $\NP(P_H)=H\cap \NP(P)$ is not included in a segment, neither is $\NP(\tilde P_H)$.   
Thus, by Proposition \ref{prop2}, 
there is   $c\in(\K^*)^n$ such that  $\Delta_{\tilde P_H} (c)=0$.  
We show that this contradicts the assumption that $P$ is quasi-ordinary.  

Let 
\begin{align*}
Q(\tilde X_1, \ldots, \tilde X_n,T, Z) = & 
T^{-d\b _{n+1}} P (( c_1 +\tilde X_1)T^{\b_1}, \ldots, (c_n +\tilde X_n)T^{\b_n},ZT^{\b_{n+1}}) \\
= & P_H(c+\tilde  X, Z) +  \sum _{m=1}^\infty P_m(\tilde X_1, \ldots, \tilde X_n, Z) T^m.  
\end{align*}
Write $(c+X)T^\b$ for $ (( c_1 +\tilde X_1)T^{\b_1}, \ldots, (c_n +\tilde X_n)T^{\b_n})$.  If $\Delta_P(X) = 
X^\alpha U(X)$ then the discriminant of $Q$ is given by 
\begin{align*}
\Delta_Q(\tilde X,T)=  T^{-d(d-1)} \Delta_P (( c +\tilde X)T^{\b}) 
= T ^M(c+\tilde X)^{\a}U((c+\tilde X)T^{\b}),
 \end{align*}
where $M=\sum_i \alpha_i \beta_i- d(d-1)$.  
Let $ Q_k (\tilde X,T,Z)= P_H(c+\tilde X,Z) +  \sum _{m=1}^{k-1} P_mT^m$.  Then 
$Q (\tilde X,T,Z) - Q_k (\tilde X,T,Z) \in (T)^k$ and hence $\Delta_Q(\tilde X,T) - \Delta_{Q_k} 
(\tilde X,T) \in (T)^{(d-1)k}$.  That means that for $k$ sufficiently large, $k(d-1) > M$, 
$\Delta_{Q_k}(\tilde X,T)$ equals $T^M U_1(\tilde X,T)$ in $\K[[\tilde X,T]]$, where $U_1(0)\neq 0$.  
Here we use the fact that all $c_i \ne 0$ and hence $(c+\tilde X)^{\a}$ is invertible.

Since $\b_i > 0$ for $i=1, \ldots ,n$, all 
$P_m$ are polynomials  and hence $Q_k$ is a polynomial.  
  By Proposition \ref{restriction},  the discriminant of  
$(P_H(c+\tilde  X, Z))_{red} = \tilde P_H(c+\tilde X, Z)$ divides $\Delta_{Q_k}$, and therefore has to be  nonzero at $\tilde  X=0$.     This contradicts the fact that $\Delta_{\tilde P_H} (c)=0$.  
This ends the proof of Theorem \ref{maintheorem}.   \qed

\begin{corollary}\label{gamma}
If  $\Delta_P (X)= X_1^{\alpha_1} \cdots X_r^{\alpha_r} U(X)$ with $r\le n$, then  there is $\gamma\in \N^r\times 0$ and $i_0 \in \{2,\ldots ,d\}$ such that 
 $a_{i_0} (X) = X^\gamma U_{i_0}(X)$, $U_{i_0}(0) \ne 0$, 
and  $X^{j\gamma}$ divides $a_j^{i_0}$ for all $j\in \{2,\ldots ,d\}$.  
\end{corollary}

\begin{proof}
By   Theorem \ref{maintheorem} there is such $\gamma\in \N^n.  $ 
If $\Delta_P (X)$ is not divisible by 
$X_k$ then there is at least one coefficient $a_i$ that is not divisible by $X_k$.  
\end{proof}


\section{2nd proof of Theorem \ref{maintheorem}.\\}\label{alternativeproof}
First we show that Theorem \ref{AJ} implies Theorem \ref{maintheorem}. This proposition is well known, see \cite{Z} for instance.

\begin{prop}\label{prop}
Let $\K$ be an algebraically closed field of characteristic zero. Let $P(Z)\in \K[[X]][Z]$ be a quasi-ordinary Weierstrass polynomial with $a_1=0$. If there is $q\in\N\setminus \{0\}$  such that $P(Z)$ has its roots in $\K[[X^{\frac{1}{q}},...,X_n^{\frac{1}{q}}]]$ then $P$ is $\nu$-quasi-ordinary.
\end{prop}
\begin{proof}
Let  $P(Z)\in\K[[X_1,...,X_n]][Z]$ be a quasi-ordinary polynomial such that its roots 
 $Z_1(X),\ldots, Z_d(X)\in\K [[X^{1/q}]]$ for some $q\in\N\setminus \{0\}$.  In what follows we assume for simplicity 
$q=1$, 
substituting the powers if necessary.  For $i\ne j$ 
$$ 
Z_{i,j}(X) = Z_i(X) - Z_j(X) = X^{\beta_{ij}} u_{ij} (X) , \qquad u_{ij}(0) \ne 0 . 
$$
For each $i$ fixed, the series $Z_{i,j}$, $j\ne i$, and their differences are normal crossings (that is monomial times a unit).  
By the lemma below,  the set $\{\b_{i,j},\, j=1,...,d\}$ is totally ordered.

\begin{lemma}\label{lemBM}(\cite{Z}, \cite{B-M} Lemma 4.7) 
Let $\alpha,\beta,\gamma \in \N^n$ and let $a(X), b(X), c(X)$ be invertible 
elements of $\K[[X]]$.  If  
$$
a(X) X^\alpha -b(X)X^\beta = c(X)X^\gamma,
$$
then either $\alpha_i \leq\beta_i$ for all $i=1,\ldots, n$ or 
$\beta_i\leq\alpha_i$ for all $i=1,\ldots, n$.
\end{lemma}

\begin{proof}[Proof of Lemma \ref{lemBM}]
For $f=\sum_{\a\in\N^n}f_{\a}X^{\a}\in\K[[X]]$, let Supp$(f):=\{\a\in\N^n\ / f_{\a}\neq 0\}$ be the support of $f$. We always have  $\Supp(c(X)X^{\g})\subset \g+\N^n$ and, since $c(X)$ is invertible, $\g\in\Supp (c(X)X^{\g})$. Since $a(X)X^{\a}-b(X)X^{b}=c(X)X^{\g}$, then $$\Supp(c(X)X^{\g})\subset\Supp(a(X)X^{\a})\cup\Supp(b(X)X^{\b})\subset (\a+\N^n)\cup(\b+\N^n).$$
Thus either $\g\in\a+\N^n$ or $\g\in\b+\N^n$. If $\g\in \a+\N^n$, then $X^{\a}$ divides $X^{\g}$, hence $b(X)X^{\b}$ is divisible by $X^{\a}$ which means that $\a\leq \b$ component-wise.
\end{proof} 

Denote 
$\beta_{i} = \min _{j\ne i} \beta_{i,j}$.  

\begin{lemma}\label{lemma1}
We have $ \beta_{1} =  \beta_{2} = \cdots = \beta_{d}$.  Denote this common exponent by $ \beta$. 
Then each $a_i$ is divisible by $(X^{ \beta})^i$.   
\end{lemma}

\begin{proof}
For $i,j,k$ distinct we have $\beta_{i,j} \ge \min \{ \beta_{i,k}, \beta_ {j,k}\}$  (with the equality if 
$ \beta_{i,k}\ne \beta_ {j,k}$).   Therefore $\beta_{i,j} \ge  \beta_k$ and hence $ \beta_i \ge  \beta_k$.   
This shows $ \beta_{1} =  \beta_{2} = \cdots = \beta_{d}$. 
 Because $a_1=0$,  
$$Z_{i} =  Z_{i} - \frac 1 d \sum_{k=1}^d Z_k = \sum_{k=1}^d \frac {Z_{i}-Z_k} d$$
is divisible by $X^{\beta}$.  
\end{proof}

To complete the proof we show that there is $i_0$ such that $a_{i_0} /X^{i_0\beta} $ 
does not vanish at the origin.    By Lemma 
\ref{lemma1} we may write 
$$
Z_i (X) = X^\beta \tilde Z_i (X) .
$$
Then $i_0$ is the number of $i$ such that $ \tilde Z_i (0) \ne 0$, and then $\gamma = i_0 \beta$. 
\end{proof}

\begin{rmk}
The set $\{\b_{i,j}\}$ determines many properties of the hypersurface germ defined by the quasi-ordinary polynomial $P$ (see for instance \cite{Gau}, \cite{Li}).
\end{rmk}

\begin{rmk}
It is possible to define a change of coordinates of the form $Z'=Z+a(X)$ in such a way that the Newton polyhedron of the quasi-ordinary polynomial $P(Z'-a(X))$ has only one compact face of dimension $\leq 1$ (see \cite{G-P2}). 
\end{rmk}

In order to prove Theorem \ref{maintheorem} we use Proposition \ref{complex}. Hence by Proposition \ref{prop},  Theorem \ref{maintheorem} is true for any $P(Z)\in\C\{X\}[Z]$.

Let $\K$ be any algebraically closed field of characteristic zero. Let $P(Z)\in\K[[X]][Z]$, $P(Z)=Z^d+a_1(X)Z^{d-1}+\cdots+a_d(X)$. Then the coefficients of the $a_i$'s are in a field extension of $\Q$ generated by countably many elements and denoted by $\K_1$. Since trdeg$_{\Q}\C$ is not countable and since $\C$ is algebraically closed, there is an embedding $\K_1\hookrightarrow \C$. Since the conditions of being quasi-ordinary and $\nu$-quasi-ordinary does not depend on the embedding $\K_1\hookrightarrow\C$,  
we may assume that $P(Z)\in\C[[ X_1,...,X_n]][Z]$.

Then let us assume that $P(Z)\in\C[[X_1,...,X_n]][Z]$ such that $a_1=0$ and $\D_P(X)=X^{\a}u(X)$ with $u(0)\neq 0$. Let us remark that $\D(X)=R(a_2(X),...,a_n(X))$ for some polynomial $R(A_2,...,A_d)\in\Q[A_2,...,A_d]$. Let us denote by $Q\in\Q[X_1,...,X_n][A_2,...,A_d,U]$ the following polynomial: 
$$Q(A_2,...,A_d,U):=\D(A_2,...,A_d)-X^{\a}U.$$ 
Then $Q(a_2(X),....,a_d(X),u(X))=0$. By the Artin Approximation Theorem (cf. \cite{Ar} Theorem 1.2), for every 
 integer $j\in\N$, there exist $a_{2,j}(X)$,..., $a_{n,j}(X)$, $u_{j}(X)\in\C\{ X_1,...,X_n\}$ such that 
$$Q(a_{2,j}(X),..., a_{n,j}(X), u_{j}(X))=0,$$
 $a_k(X)-a_{k,j}(X)\in(X)^j$ and $u(X)-u_{j}(X)\in (X)^j$. Let us denote
$$P_j(Z):=Z^d+a_{2,j}(X)Z^{d-2}+\cdots+a_{d,j}(X)\in\C\{ X_1,...,X_n\}[Z].$$ Then $P_j(Z)$ is quasi-ordinary for $j\geq 1$ and $\NP(P)\subset \NP(P_j)+\N^n_{\geq j}$ where 
$$\N^n_{\geq j}:=\{k\in\N^n\ /\ k_1+\cdots+k_n\geq j\}.$$

If $P(Z)$ were not $\nu$-quasi-ordinary then $\NP(P)$ would have a compact face of dimension at least 2 and containing the point $(0,...,0,d)$. For $j>j_0:=\max |\g|$ where $\g$ runs through the vertices of $\NP(P)$, we see that this compact face is also a face of $\NP(P_j)$ and this contradicts the fact that $P_j(Z)$ is $\nu$-quasi-ordinary.
 Thus Theorem \ref{maintheorem} is proven.\\

In fact, by using the Strong Artin Approximation Theorem, we can prove the following result about the continuity of the Newton polyhedra of $P(Z)$ with respect to its discriminant.

\begin{prop}
For any $d\in\N$ and any $\a\in\N^n$, there exists a function $\b:\N\lgw \N$ satisfying the following property:  for any $k\in\N$  and any Weierstrass polynomial $P(Z)=Z^d+a_1Z^{d-1}+\cdots+a_d\in\K[[X_1,...,X_n]][Z]$ of degree $d$ such that $a_1=0$ and its discriminant $\D_P =X^{\a}U(X)$ mod. $(X)^{\b(k)}$ there exists a compact edge $S$ containing $R_0:=(0,...,0,d)$ such that one has $\NP(P)\subset |S|+\N_{\geq k}^n$.
\end{prop}

\begin{proof}
Let $d\in\N$ and $\a\in\N^n$.  Let us denote by $Q\in\Q[[X_1,...,X_n]][A_2,...,A_d,U]$ the  polynomial
$Q(A_2,...,A_d,U):=\D(A_2,...,A_d)-X^{\a}U$
where $\D(A) $ is the discriminant of the polynomial $Z^d+A_2Z^{d-2}+\cdots+A_d$. 

By the Strong Artin Approximation Theorem (cf. \cite{Ar2} Theorem 6.1), there exists a function $\b:\N\lgw \N$ such that for any $k\in\N$ and any $a_2,...,a_d,u\in\K[[X]]$ with $Q(a_2,...,a_d,u)\in(X)^{\b(k)}$ there exist $\ovl{a}_2$,..., $\ovl{a}_d$, $\ovl{u}\in\K[[X]]$ such that $Q(\ovl{a}_2,...,\ovl{a}_d,\ovl{u})=0$ and $\ovl{a}_i-a_i$, $\ovl{u}-u\in (X)^k$ for all $i$.

Let $P(Z)=Z^d+a_2Z^{d-2}+\cdots+a_d$ such that $\D(P)=X^{\a}U(X)$ mod. $(X)^{\b(k)}$. Then there exists a polynomial $\ovl{P}(Z)$ such that its discriminant $\ovl{\D}(\ovl{P})=X^{\a}\ovl{U}(X)$ with $\ovl{U}(0)\neq 0$ and $P(Z)-\ovl{P}(Z)\in (X)^k$. By Theorem \ref{maintheorem} $NP(\ovl{P})\subset |S|$ hence $NP(P)\subset |S|+\N_{\geq k}^n$.
\end{proof}


\section{Applications}\label{applications}

\subsection{Proof of the Abhyankar-Jung Theorem. }\label{AJproof}
 Let 
$$P(Z)=Z^d+a_1(X) Z^{d-1}+\cdots+a_d(X),$$ 
$a_i\in \K[[X]]$, and let $\K$ be an algebraically closed field of characteristic zero.  We suppose that the discriminant 
of $P$ is of the form $X^\alpha U(X)$, $U(0)\ne 0$.  It is not necessary to suppose that all $a_i(0)=0$ (of course 
Theorem \ref{maintheorem} holds if one of the $a_i$'s does  not vanish at the origin).  The procedure consists of a number of steps 
simplifying the polynomial and finally factorising it to two polynomials of smaller degree.  Theorem \ref{maintheorem} is used in Step 2.    \\

\textbf{Step 1.} (Tschirnhausen transformation) 
Replace $Z$ by $Z-\frac{a_1(X)}{d}$.   The  coefficients  
$a= (a_1, a_2, \ldots, a_d)$ are replaced by $(0, \tilde a_2, \ldots, \tilde a_d)$ so we can assume $a_1=0$. \\

\textbf{Step 2.}
Write  $Z= X^\beta \tilde Z$, and divide each $a_i$ by $X^{i\beta}$, where $\beta = 
\gamma/i_0$ for $\gamma$ and $i_0$ given by Corollary \ref{gamma} .  If  the coordinates of $\beta$ are not integers 
this step involves a substitution of powers.  Then 
\begin{align}\label{step2}
P(Z) = P(X^\beta \tilde Z) = X^{d\beta} (\tilde Z^d+ \tilde a_1(X)  \tilde Z^{d-1}+\cdots+\tilde a_d(X) ), 
\end{align}
where $\tilde a_i = a_i/X^{i\beta}$. We replace $P(Z)$  by  $\tilde P(Z) = Z^d+ \tilde a_1(X)  Z^{d-1}+\cdots+\tilde a_d(X)$.   \\

\textbf{Step 3.} Now $\tilde a_{i_0}\ne 0$ and since $\tilde a_1=0$  the polynomial $ Q(Z) = \tilde P(Z)|_{X=0}\in \K[Z]$ has at least two distinct roots in $\K$ and can be factored 
$Q(Z)=Q_1(Z) Q_2(Z)$,  $d_i:= \deg Q_i<d, i=1,2$, 
 where $Q_1(Z)$ and $Q_2(Z)$ are two polynomials of $\K[Z]$ without common root. 
  \\
 
 \textbf{Step 4.} 
By the Implicit Function Theorem there is a factorisation $\tilde P(Z)= \tilde P_1(Z) \tilde P_2(Z)$ with ${\tilde P_i}(Z)_{|X=0}=Q_i(Z)$ for $i=1,2$.  
More precisely, let 
\begin{align*}
& q(Z) =Z^d+a_1 Z^{d-1}+\cdots+a_d, \\
&  q_1(Z)=Z^{d_1}+b_1 Z^{d_1-1}+\cdots+b_{d_1},  \quad q_2(Z)=Z^{d_2}+c_1 Z^{d_2-1}+\cdots+c_{d_2}
\end{align*}
where $a=(a_1, \ldots ,a_d)\in \K^{d}$,  $b=(b_1,\cdots, b_{d_1})\in \K^{d_1}$,  
$c=(c_1,\cdots, c_{d_2})\in \K^{d_2}$.  
The product of polynomials $q=q_1q_2$ defines  a map $a= \Phi (b,c)$, $ \Phi :\K^d \to \K^d$,  that is polynomial  in $b$ and $c$.   
  The  Jacobian determinant of $\Phi$ equals the resultant of $q_1$ and $q_2$.   Denote by $b_0,c_0$ the coefficient   vectors of $Q_1$ and $Q_2$ and consider $\tilde \Phi : \K[[X]]^d \to \K[[X]]^d$ given by 
 $\tilde \Phi (b,c) = \Phi (b+b_0, c+c_0) - \tilde a(0)$.  
Then  the Jacobian determinant of  $\tilde \Phi$   is invertible and hence, by the Implicit Function Theorem
 for formal power series, the inverse of $\tilde \Phi$  is a well-defined power series.  
 Define  $(b(X), c(X))\in \K[[X]]^d$ as  $\tilde \Psi ^{-1} (\tilde a(X)-\tilde a(0)) + (b_0,c_0)$.  Then  $\tilde P(Z)=\tilde  P_1(Z)\tilde P_2(Z)$ where $\tilde P_1(Z)=Z^{d_1}+b_1(X) Z^{d_1-1}+\cdots+b_{d_1}(X)$ and $\tilde P_2(Z)=Z^{d_2}+c_1(X) Z^{d_2-1}+\cdots+c_{d_2}(X)$.\\

We may describe the outcome of Steps 2-4 by the following. 
 Denote the new polynomial obtained in Step 2  
by $\tilde P (\tilde Z)$, where $\tilde Z= Z/X^\beta$, 
\begin{align*}
\tilde P(\tilde Z) = \tilde Z^d+ \tilde a_1(X)  \tilde Z^{d-1}+\cdots+\tilde a_d(X) .  
\end{align*}
Then by Step 4 we may factor $\tilde P = \tilde P_1 \tilde P_2$,  $d_1 = 
\deg \tilde P_1<\deg P$, $d_2 =\deg \tilde P_1<\deg P$,   and 
\begin{align}\label{factorisation}
P(Z) = X^{d\beta} \tilde P (\tilde Z) = X^{d_1\beta} \tilde P_1 (\tilde Z) X^{d_2\beta} \tilde P_2 (\tilde Z)  = 
P_1(Z) P_2(Z).  
\end{align}

The discriminant of $P$ is equal to the product of the discriminants  of $P_1$ and $P_2$ and the square of the resultant of $P_1$ and $P_2$.  
Hence $\Delta_{P_1}$, and similarly $\Delta_{P_2}$, is equal to a monomial times a unit.  
Thus we continue the procedure for $P_1(Z)$ and $P_2(Z)$ until we reduce to polynomials of degree one.  
This ends the proof. 
 \qed\\

Note that if $\K$ is a field of characteristic zero not necessarily algebraically closed then in Step 3 we may need a finite field extension.  Thus we obtain the following result, see \cite {L} the last page proposition.  

\begin{prop}\label{lastprop}
 Let $P \in \K[[X_1,\ldots , X_n]][Z] $ be a quasi-ordinary Weierstrass polynomial with coefficients in a field of characteristic zero 
 (not necessarily algebraically closed).  
 Then there is a finite extension $\K'\supset \K$ such that the roots of $P(Z)$ are in 
$\ \K' [ [X^{\frac 1 q}]]$ for some $q \ge 1$.    
\end{prop}

\begin{rmk}
It is not true in general that the roots of $\nu$-quasi-ordinary Weierstrass 
polynomials are Puiseux series in several variables. In the latter algorithm, its is not difficult 
to check that in \eqref{factorisation}, $P_1$ and $P_2$ satisfy property (1) of the definition of $\nu$-quasi-ordinary polynomials but not  property (2). For example let
$$P(Z):=Z^4-2X_1X_2(1-X_1-X_2)Z^2+(X_1X_2)^2(1+X_1+X_2)^2.$$
This polynomial is $\nu$-quasi-ordinary and factors as
$$P(Z)=\left(Z^2+2(X_1X_2)^{\frac{1}{2}}Z+X_1X_2(1+X_1+X_2)\right)\left(Z^2-2(X_1X_2)^{\frac{1}{2}}Z+X_1X_2(1+X_1+X_2)\right)$$
in $\K[[X^{\frac{1}{2}}]][Z]$. One has
$$Z^2+2(X_1X_2)^{\frac{1}{2}}Z+X_1X_2(1+X_1+X_2)=(Z+(X_1X_2)^{\frac{1}{2}})^2+X_1X_2(X_1+X_2).$$
This shows that $P(Z)$ is irreducible in $\K[[X]][Z]$ and that none of its roots is a Puiseux series  
in $X_1,X_2$ (all roots 
are branched along $X_1+X_2=0$).  

\end{rmk}


\subsection{Abhyankar-Jung Theorem for Henselian subrings of $K[[X]]$.}

Consider  Henselian subrings of $\K[[X]]$ which do not necessarily have the Weierstrass division property.

\begin{definition}\label{def}
We will consider $\K\{\{ X_1,\ldots ,X_n\}\}$  a subring of $\K[[X_1,\ldots ,X_n]]$ such that:
\begin{enumerate}
\item[i)] $\K\{\{X_1,\ldots ,X_n\}\}$ contains $\K[X_1,...,X_n]$.
\item[ii)] $\K\{\{ X_1,\ldots ,X_n\}\}$ is a Henselian local ring with maximal ideal generated by $X_1,\ldots , X_n$.
\item[iii)] $\K\{\{X_1,\ldots ,X_n\}\}\cap(X_i)\K[[X_1,\ldots ,X_n]]=(X_i)\K\{\{ X\}\}$.
\item[iv)] If $f\in\K\{\{X\}\}$ then $f(X_1^{e_1},\ldots ,X_n^{e_n})\in\K\{\{X\}\}$ for  any $e_i\in \N\setminus \{0\}$. 
\end{enumerate}
\end{definition}

\begin{ex}
The rings of algebraic or formal power series over a field satisfy Definition \ref{def}. If $\K$ is a valued field, then the ring of convergent power series over $\K$ satisfies also this definition. The ring of germs of  quasi-analytic functions over $\R$ also satisfies this definition (even if there is no Weierstrass Division Theorem in this case, see \cite{C} or \cite{ES}).  We come back to this example in the next subsection.  
\end{ex}

Since the Implicit Function Theorem holds for such rings (they are Henselian) we obtain by the procedure  of 
Section \ref{AJproof} the following result.  

\begin{theorem}\label{henselian}
Let $\K\{\{ X_1,\ldots ,X_n\}\}$ be a subring of $\K[[X_1,\ldots ,X_n]]$ like in Definition \ref{def}. Moreover let us assume that $\K$ is an algebraically closed field of characteristic zero. Let $P(Z)\in\K\{\{X\}\}$ be a quasi-ordinary Weierstrass polynomial such that its discriminant, 
$$\Delta (X)= X_1^{\alpha_1} \cdots X_r^{\alpha_r} U(X),$$ $r\le n$, where $\alpha_i $ are positive integers 
 and $U(0)\ne 0$.  Then there exists an integer $q\in\N\setminus \{0\}$ such that the roots of $P(Z)$ are in $\K\{\{X_1^{\frac{1}{q}},\ldots ,X_r^{\frac{1}{q}},X_{r+1},\ldots ,X_n\}\}$.
\end{theorem}


\subsection{Quasi-analytic functions.}\label{coeffs}
Denote by  $\mathcal E_n$  the algebra of complex valued $C^\infty$ germs of 
$n$ real variables: $f:(\R^n,0)\to  \C$.  We call a subalgebra $\mathcal C_n \subset \mathcal E_n$ \emph{quasi-analytic}  
if the Taylor series morphism $\mathcal C_n \to \C [[ X_1, \ldots, X_n]]$ is injective.  
If this is the case we identify $\mathcal C_n$ with its image in $\C [[X_1, \ldots, X_n]]$.  
Usually one considers families of algebras $\mathcal C_n$ defined for all $ n \in \N$ and satisfying some 
additional properties, such as stability by differentiation, taking implicit functions, composition, etc.,  see \cite {T}.  

If  $\mathcal C_n$ is Henselian in the sense of Definition \ref{def}, that is practically always the case,  then we may apply Theorem \ref{henselian}.  Since the arguments of quasi-analytic functions are real, the substitution of powers $X_i=Y_i^{\gamma_i}$ is not surjective if one of $\gamma_i$'s is even.  Thus for the sake of applications, 
cf. \cite {R},  it is natural to consider the power substitutions with signs $X_i=\varepsilon_i Y_i^{\gamma_i}$, 
$\varepsilon_i  = \pm 1$.  Thus Theorem \ref{henselian} implies the following.

\begin{theorem}[Abhyankar-Jung Theorem for quasi-analytic germs]\label{AJqa}
Let $\mathcal C_n$ be a quasi-analytic algebra satisfying Definition \ref{def}.  Let the discriminant 
$\Delta_P (X)$ of 
\begin{align*}
P(Z)=Z^d+a_1(X_1, \ldots , X_n) Z^{d-1}+\cdots+a_d(X_1, \ldots , X_n) \in \mathcal C_n [Z]
\end{align*}
 satisfy 
$\Delta_P (X)= X_1^{\alpha_1} \cdots X_r^{\alpha_r} U(X)$, 
where $U(0)\ne 0$, and $r\le n$.  
Then there is $\gamma \in (\N\setminus \{0\})^r$ such that for every combination of signs 
$\varepsilon \in \{-1,1\}^r$,   the polynomial 
\begin{align*}
Z^d+a_1(\varepsilon _1 Y_1^{\gamma_1}, \ldots ,\varepsilon _r Y_r^{\gamma_r}, Y_{r+1}, \ldots , Y_n) 
 Z^{d-1}+\cdots+a_d(\varepsilon _1 Y_1^{\gamma_1}, \ldots ,\varepsilon _r Y_r^{\gamma_r}, Y_{r+1}, \ldots , Y_n) 
\end{align*}
has $d$ distinct roots in $\mathcal C_n$.  
\end{theorem}


\section{Toric Case}\label{toric}
We thank Pedro Gonz\'alez P\'erez who pointed out that our proof of Abhyankar-Jung Theorem may be generalised to the toric case. This is the aim of this section.

Let $\s\subset\R^n$ be a rational strictly convex polyhedral cone of dimension $d$. Let 
$$\s^{\vee}:=\{v\in (\R^n)^*\ /\ \langle v,u\rangle\geq 0,\ \forall u\in\s\}$$
be the dual cone of $\s$. Let $V_{\s}:=\text{Spec}\left(\K[X^v\ /\ v\in\s^{\vee}\cap\Z^d]\right)$ the associated affine toric variety. The ideal $\m_{\s}$ generated by the $X^v$, when $v$ runs through $\s^{\vee}\cap(\Z^d)^*$, is a maximal ideal defining a closed point of $V_{\s}$ denoted by $0$. In fact $\K[X^v\ /\ v\in\s^{\vee}\cap\Z^d] \simeq \K[Y]/I$ where $Y=(Y_1,...,Y_m)$ for some integer $m$ and $I$ is a binomial ideal. In this case $\m_{\s}\simeq (Y)$.

When $\K=\C$ we  define $\O_{V_{\s},0}:=\C\{X^v\}_{v\in\s^{\vee}\cap\Z^d}\simeq\C\{Y\}/I$ the ring of germs of analytic functions at $(V_{\s},0)$.

Let $P(Z)=Z^d+a_1Z^{d-1}+\cdots+a_d\in\K[[X^v]]_{v\in\s^{\vee}\cap\Z^d}[Z]$ be a toric polynomial. The polynomial $P(Z)$ is called \textit{quasi-ordinary} if  its discriminant equals $X^{\a}U(X)$, $\a\in\s^{\vee}$ and $U(X)$ being a unit of $\K[[X^v]]_{v\in\s^{\vee}\cap\Z^d}$. We call $P(Z)$ \textit{a Weierstrass polynomial} if $a_i\in\m_{\s}$ for $i=1,...,d$. In \cite{G-P1}, P. Gonz\'alez P\'erez proved the following theorem:
\begin{theorem}\cite{G-P1}\label{GP}
Let $P(Z)\in\C\{X^v\}_{v\in\s^{\vee}\cap\Z^d}[Z]$ be a toric quasi-ordinary polynomial. Then there exists $q\in\N$ such that $P(Z)$ has its roots in $\C\{X^v\}_{v\in\s^{\vee}\cap\frac{1}{q}\Z^d}$.
\end{theorem}

Here we will prove a generalisation of this result over any algebraically closed field $\K$ of characteristic zero. 

\begin{theorem}[Toric Abhyankar-Jung Theorem]\label{toricAJ}
Let $\K\{\{ X_1,\ldots ,X_n\}\}$ be a subring of $\K[[X_1,\ldots ,X_n]]$ like in Definition \ref{def}. Moreover let us assume that $\K$ is an algebraically closed field of characteristic zero. Let  
$P \in \K\{\{X^v\}\}_{v\in\s^{\vee}\cap\Z^d}[Z]$ be a toric quasi-ordinary Weierstrass polynomial. 
Then there is $q\in\N\setminus \{0\}$ such that $P(Z)$ has its roots in 
$\K\{\{X^v\}\}_{v\in\s^{\vee}\cap\frac{1}{q}\Z^d}$.
\end{theorem} 

First we define  $\nu$-quasi-ordinary polynomials in the toric case.  
Given $P\in \K[[X^v,Z]]_{v\in\s^{\vee}\cap\Z^d}$.  Write 
$$
P(X,Z) = \sum_{(i_1, \ldots, i_{n+1})} P_{i_1, \ldots, i_{n+1}} X^{i_1} \cdots X^{i_n} Z^{i_{n+1}}.
$$
Let $H(P) = \{ (i_1, \ldots, i_{n+1}) \in (\s^{\vee}\cap\Z^d)\times\N; P_{i_1, \ldots, i_{n+1}} \ne 0\}$.  The Newton polyhedron of $P$ 
is the convex hull in $\R^{n+1}$ of $\bigcup_{a\in H(P)} (a+(\s^{\vee}\cap\Z^d)\times\R_{\geq 0})$, and we will denote it by $\NP (P)$.  

A Weierstrass polynomial as before is called  
\emph {$\nu$-quasi-ordinary} if there is a point $R_1$ of the Newton polyhedron $\NP (P)$, 
$R_1\ne R_0=(0,\ldots, 0,d)$, such that if $R_1'$ denotes the projection of $R_1$ onto $\R^n\times 0$ 
from $R_0$, and $S=|R_0,R_1'|$ is the segment joining $R_0$ and $R'_1$, then 
\begin{enumerate}
\item
$\NP (P) \subset |S| = \bigcup_{s\in S} (s+(\s^{\vee}\cap\Z^d)\times\R_{\geq 0})$
\item
$P_S= \sum_{(i_1, \ldots, i_{n+1})\in S} P_{i_1, \ldots, i_{n+1}}  X_1^{i_1} \cdots X_n^{i_n} Z^{i_{n+1}}$ is not a power of a 
linear form in $Z$.  
\end{enumerate} 
The second condition is satisfied automatically if $a_1=0$.   
The proof of the following lemma is the same as the proof of Lemma \ref{lemma3.1}. 

\begin{lemma}
Let $P(Z)\in \K[[X^v]]_{v\in\s^{\vee}\cap\Z^d}[Z]$ be a Weierstrass polynomial \eqref{polynomial} such that $a_1=0$.  
The following conditions are equivalent: 
\begin{enumerate}
\item
$P$ is $\nu$-quasi-ordinary.
\item
$\NP(P)$ has only one compact edge containing $R_0$.  
\item
The ideal $(a_i^{d!/i}(X))_{i=2,\ldots ,d}\subset \K[[X^v]]_{v\in\s^{\vee}\cap\Z^d}$ is monomial and principal.  
\end{enumerate} 
\end{lemma}

In order to prove Theorem  \ref{toricAJ} we first show the toric version of Theorem \ref{maintheorem}.  

\begin{theorem}
Let $\K$ be an algebraically closed field of characteristic zero and let 
$P \in \K[[X^v]]_{v\in\s^{\vee}\cap\Z^d}[Z]$ be a toric quasi-ordinary Weierstrass polynomial such that  $a_1=0$.   
Then  the ideal $(a_i^{d!/i})_{i=2,\ldots ,d}\subset \K[[X^v]]_{v\in\s^{\vee}\cap\Z^d}$ is principal and generated by a monomial.    
\end{theorem}

As before, this theorem may be reformulated as the following:

\begin{theorem}\label{theo'}
If $P$ is a toric quasi-ordinary Weierstrass polynomial with $a_1=0$ then $P$ is  $\nu$-quasi-ordinary.
\end{theorem}

\begin{prop}\label{prop'}
Let $\K$ be an algebraically closed field of characteristic zero. Let $P(Z)\in \K[[X^v]]_{v\in\s^{\vee}\cap\Z^d}[Z]$ be a quasi-ordinary Weierstrass polynomial with $a_1=0$. If there is $q\in\N\setminus \{0\}$  such that $P(Z)$ has its roots in $\K[[X^v]]_{v\in\s^{\vee}\cap\frac{1}{q}\Z^d}$ then $P$ is $\nu$-quasi-ordinary.
\end{prop}

\begin{proof}[Proof of Proposition \ref{prop'}]
The proof of Proposition \ref{prop'} is exactly the same as the proof of Proposition \ref{prop}. We only need the following lemma:

\begin{lemma} \label{lemmaBM}
Let  $\alpha,\beta,\gamma \in \s^{\vee}\cap\Z^d$ and let $a(X), b(X), c(X)$ be invertible 
elements of $\K[[X^v]]_{v\in\s^{\vee}\cap\Z^d}$.  If  
$$
a(X) X^\alpha -b(X)X^\beta = c(X)X^\gamma,
$$
then either $\alpha\in\beta+\s^{\vee}$ or $\beta\in\alpha+\s^{\vee}$.   
\end{lemma}

The proof of Lemma \ref{lemmaBM} is exactly the same as the proof of Lemma \ref{lemBM}, we only need to replace $\N^n$ by $\s^{\vee}$.
\end{proof}

\begin{proof}[Proof of Theorem \ref{theo'}]
The proof of Theorem \ref{theo'} is  similar to the second proof of Theorem \ref{maintheorem},  section \ref{alternativeproof}.  
We replace Proposition \ref{prop} by Proposition \ref{prop'},  $\K[[X]]$ by $\K[[X^v]]_{v\in\s^{\vee}\cap\Z^d}$ 
and $\C\{X\}$ by $\C\{X^v\}_{v\in\s^{\vee}\cap\Z^d}\simeq\C\{Y\}/I$.  Then we use the fact that the ring $\C\{Y\}/I$ satisfies the Artin Approximation Theorem (cf. \cite{Ar} Theorem 1.3).
\end{proof}

\begin{proof}[Proof of Theorem \ref{toricAJ}]
We prove Theorem \ref{toricAJ} exactly as we proved Theorem \ref{AJ}.  Step 1, Step 2 and Step 3 are exactly the same (we just replace $P(Z)|_{X=0}$ by the image of $P(Z)$ in $\K\{\{X^v\}\}_{v\in\s^{\vee}\cap\Z^d}[Z]/\m_{\s}$). For Step 4, we replace the Implicit Function Theorem by Hensel's Lemma (cf. \cite{EGA} 18.5.13) since $\K\{\{X^v\}\}_{v\in\s^{\vee}\cap\Z^d}\simeq\K\{\{Y\}\}/I$ is a local Henselian ring.
\end{proof}


\end{document}